\documentclass[final, 12pt,english]{amsart}
\usepackage{amsmath,amsthm,amssymb,amsfonts,amscd, array,latexsym}
\usepackage{graphicx}
\usepackage[T1]{fontenc}
\usepackage{amsfonts,amsmath,amssymb}
\usepackage{tikz}
\usetikzlibrary{positioning, calc, trees, decorations.markings, patterns, arrows, arrows.meta}
\usepackage{diagmac2}
\usepackage[notref,notcite]{showkeys} %my

\usepackage[noadjust]{cite}

\newcommand{\BB}{\mathbf{B}}
\newcommand{\RR}{\mathbb{R}}

\newcommand{\levdist}{1.2cm}

\makeatletter

\renewcommand{\p@enumii}{}
\makeatother

\DeclareMathOperator{\card}{card}
\DeclareMathOperator{\diam}{diam}

\newtheorem{theorem}{Theorem}[section]
\newtheorem*{theorem*}{Theorem}
\newtheorem{corollary}[theorem]{Corollary}
\newtheorem{corollary*}{Corollary}
\newtheorem{lemma}[theorem]{Lemma}
\newtheorem{proposition}[theorem]{Proposition}
\newtheorem{conjecture}[theorem]{Conjecture}

\theoremstyle{definition}
\newtheorem{definition}[theorem]{Definition}
\newtheorem{example}[theorem]{Example}
\newtheorem{problem}[theorem]{Problem}

\theoremstyle{remark}
\newtheorem{remark}[theorem]{Remark}
\numberwithin{equation}{section}

\newcommand{\acr}{\newline\indent}

\begin{document}

\title[Locally finite ultrametric spaces and labeled trees]{Locally finite ultrametric spaces \\ and labeled trees}
\author{Oleksiy Dovgoshey and Alexander Kostikov}

\address{\textbf{Oleksiy Dovgoshey}\acr
Institute of Applied Mathematics and Mechanics \acr
of the NAS of Ukraine, Sloviansk, Ukraine;\acr
Department of Mathematics and Statistics \acr
University of Turku, Turku, Finland}
\email{oleksiy.dovgoshey@gmail.com}

\address{\textbf{Alexander Kostikov}\acr
LLC Technical  University
``Metinvest Politechnic'' \acr
Zaporizhia, Ukraine}
\email{alexkst63@gmail.com}

\begin{abstract}
It is shown that a locally finite ultrametric space \((X, d)\) is generated by labeled tree if and only if, for every open ball \(B \subseteq X\), there is a point \(c \in B\) such that \(d(x, c) = \diam B\)  whenever \(x \in B\) and \(x \neq c\). For every finite ultrametric space $Y$ we construct an ultrametric space $Z$ having the smallest possible number of points such that $Z$ is generated by labeled tree and $Y$ is isometric to a subspace of $Z$. It is proved that for a given $Y$, such a space $Z$ is unique up to isometry.
\end{abstract}



\maketitle

\section{Introduction}

In what follows, we will denote by \(\RR^{+}\) the half-open interval \([0, \infty)\).

A \textit{metric} on a set $X$ is a function $d\colon X\times X\rightarrow \RR^+$ such that for all \(x\), \(y\), \(z \in X\)
\begin{enumerate}
\item $d(x,y)=d(y,x)$,
\item $(d(x,y)=0)\Leftrightarrow (x=y)$,
\item \(d(x,y)\leq d(x,z) + d(z,y)\).
\end{enumerate}

A metric space \((X, d)\) is \emph{ultrametric} if the \emph{strong triangle inequality}
\[
d(x,y)\leq \max \{d(x,z),d(z,y)\}.
\]
holds for all \(x\), \(y\), \(z \in X\). In this case the function \(d\) is called \emph{an ultrametric} on \(X\).

\begin{definition}\label{d1.1}
Let \((X, d)\) and \((Y, \rho)\) be metric spaces. A mapping \(\Phi \colon X \to Y\) is an \emph{isometric embedding}  if
\[
d(x,y) = \rho(\Phi(x), \Phi(y))
\]
holds for all \(x\), \(y \in X\). A bijective isometric embedding is said to be \emph{isometry}. Metric spaces are \emph{isometric} if there is an isometry of these spaces.
\end{definition}

Let \((X, d)\) be a metric space. An \emph{open ball} with a \emph{radius} \(r > 0\) and a \emph{center} \(c \in X\) is the set
\[
B_r(c) = \{x \in X \colon d(c, x) < r\}.
\]

\begin{definition}\label{d1.1_1}
A metric space \((X, d)\) is called \emph{locally finite} if card $B$ is finite for every open ball \( B \subseteq X \).
\end{definition}

In addition to open balls, we also need some other subsets of \((X, d)\), which we will call the centered spheres.

\begin{definition}\label{d1.2}
Let \((X, d)\) be a metric space. A set \(C \subseteq X\) is a \emph{centered sphere} in \((X, d)\) if there are \(c \in C\), the center of \(C\), and \(r \in \RR^+\), the radius of \(C\), such that
\begin{equation}\label{d1.2:e1}
C = \{x \in X \colon d(x, c) = r\} \cup \{c\}.
\end{equation}
\end{definition}

Equality \eqref{d1.2:e1} means that \(C\) is the sphere \(\{x \in X \colon d(x, c) = r\}\) with the added center \(c\).

We denote by \(\mathbf{B}_X = \mathbf{B}_{X, d}\) and  \(\mathbf{Cs}_X = \mathbf{Cs}_{X, d}\) the sets of all open balls of the metric space \((X, d)\) and, respectively, the set of all centered spheres of this space.

\begin{definition}\label{d1.3}
A \emph{labeled tree} is a pair \((T, l)\), where \(T\) is a tree and \(l\) is a mapping defined on the vertex set \(V(T)\).
\end{definition}

In what follows, we will consider only the nonnegative real-valued labelings \(l\colon V(T)\to \RR^{+}\).

 Following~\cite{Dov2020TaAoG}, we define a mapping \(d_l \colon V(T) \times V(T) \to \RR^{+}\) as
\begin{equation}\label{e1.1}
d_l(u, v) = \begin{cases}
0 & \text{if } u = v,\\
\max\limits_{w \in V(P)} l(w) & \text{if } u \neq v,
\end{cases}
\end{equation}
where \(P\) is the path joining \(u\) and \(v\) in \(T\).

\begin{theorem}[\cite{DK2022AC}]\label{t1.4}
Let \(T = T(l)\) be a labeled tree. Then the function \(d_l\) is an ultrametric on \(V(T)\) if and only if the inequality
\begin{equation*}%\label{t11.9:e1}
\max\{l(u), l(v)\} > 0
\end{equation*}
holds for every edge \(\{u, v\}\) of \(T\).
\end{theorem}

Let us introduce a class \(\mathbf{UGVL}\) (Ultrametrics Generating by Vertex Labelings) by the rule: An ultrametric space \((X, d)\) belongs to \(\mathbf{UGVL}\) if and only if there is a labeled tree \(T=T(l)\) satisfying \(X = V(T)\) and \(d(x, y) = d_l(x, y)\) for all \(x\), \(y \in X\). If \((X, d) \in \mathbf{UGVL}\), then we say that \((X, d)\) is generated by labeled tree or that \((X, d)\) is an \(\mathbf{UGVL}\)-space.

The following conjecture was formulated in \cite{DK2022AC}.

\begin{conjecture}\label{con1.5}
Let \((X, d)\) be a nonempty totally bounded ultrametric space. If all points of $X$ are isolated, then the following statements are equivalent:
\begin{enumerate}
\item\label{con1.5:s1} \((X, d) \in \mathbf{UGVL}\).
\item\label{con1.5:s2} \(\BB_{X, d} \subseteq \mathbf{Cs}_{X, d}\).
\end{enumerate}
\end{conjecture}

E. Petrov proved in \cite{Pet} the validity of the conjecture for finite ultrametric spaces using some other terms and the technique of Gurvich--Vyalyiy representing trees. We repeat this result in Theorem~\ref{t4.4} of the fourth section of the paper.

In Theorem~\ref{t4.5} it is shown that the equivalence
$$
\left( (X,d)\in \mathbf{UGVL} \right) \Leftrightarrow \left( \mathbf{B}_{X,d} \subseteq \mathbf{Cs}_{X,d} \right)
$$
is valid for all nonempty locally finite ultrametric spaces \( (X,d)\).

Theorem~\ref{t4.8} shows that \( \mathbf{Cs}_{X,d}  \subseteq \mathbf{B}_{X,d} \) holds if and only if $d$ is a discrete metric on $X$.

In Theorem~\ref{t5.4},  we construct the ``minimal'' \textbf{UGVL}-extensions of arbitrary finite ultrametric space and prove that all such minimal extensions are isometric.

\section{Preliminaries. Trees and complete multipartite graphs}

A \textit{simple graph} is a pair \((V,E)\) consisting of a nonempty set \(V\) and a set \(E\) whose elements are unordered pairs \(\{u, v\}\) of different points \(u\), \(v \in V\). For a graph \(G = (V, E)\), the sets \(V=V(G)\) and \(E = E(G)\) are called \textit{the set of vertices} and \textit{the set of edges}, respectively. We say that \(G\) is \emph{empty} if \(E(G) = \varnothing\). A graph \(G\) is \emph{finite} if \(V(G)\) is a finite set. A graph \(H\) is, by definition, a \emph{subgraph} of a graph \(G\) if the inclusions \(V(H) \subseteq V(G)\) and \(E(H) \subseteq E(G)\) are valid. In this case we simply write $H \subseteq G$.

A \emph{path} is a finite nonempty graph \(P\) whose vertices can be numbered so that
\begin{align*}
V(P) &= \{x_0,x_1, \ldots,x_k\},\ k \geqslant 1, \\
E(P) & = \bigl\{\{x_0, x_1\}, \ldots, \{x_{k-1}, x_k\}\bigr\}.
\end{align*}
In this case we say that \(P\) is a path joining \(x_0\) and \(x_k\).

A graph \(G\) is \emph{connected} if for every two distinct \(u\), \(v \in V(G)\) there is a path in \(G\) joining \(u\) and \(v\).

A finite graph \(C\), with \(\card V(G) \geqslant 3\), is a \emph{cycle} if there is an enumeration of its vertices without repetitions such that
\begin{align*}
V(C) & = \{v_1, \ldots, v_n\}, \\
E(C) & = \{\{v_1, v_2\}, \ldots, \{v_{n-1}, v_{n}\}, \{v_{n}, v_{1}\}\}.
\end{align*}

\begin{definition}\label{d2.1}
A connected graph without cycles is called a \emph{tree}.
\end{definition}

A tree $T$ may have a distinguished vertex $r$ called the \emph{root} in this case $T=T(r)$ is called a \emph{rooted tree}.

\begin{definition}\label{d2.2}
If $u$ and $v$ are vertices of a rooted tree $T=T(r)$, then $u$ is a \emph{successor} of $v$
 if the path $P \subseteq T$ joining $u$ and $r$ contains the node $v$. A successor $u$ of a node $v$ is said to be a \emph{direct successor} of the node $v$ if $\{u,v\} \in E(T)$ holds.
 \end{definition}

Let $T = T(r)$ be a rooted tree and let $v$ be a node of $T$. Denote by $\delta^+(v)$ the \emph{out-degree} of $v$, i.e., $\delta^+(v)$ is the number of direct successors of $v$.  The root $r$ is a leaf of $T$ if and only if $\delta^+(r) \leqslant 1$. Moreover, for a vertex $v$ different from the root $r$, the equality $\delta^+(v) = 0$ holds if and only if $v$ is leaf of $T$.

Recall the definition of the isomorphic rooted trees.

\begin{definition}\label{d1.8_from7}
Let \(T_1 = T_1(r_1)\) and \(T_2 = T_2(r_2)\) be rooted trees. A bijection \(f \colon V (T_1) \to V (T_2)\) is an \emph{isomorphism of the rooted trees} of \(T_1\) and \(T_2\) if \(f(r_1) = r_2\) and
\begin{equation*}
(\{u, v\} \in E(T_1)) \Leftrightarrow (\{f(u), f(v)\} \in E(T_2)).
\end{equation*}
 The rooted trees \(T_1\) and \(T_2\) are isomorphic if there exists an isomorphism \(f \colon V (T_1) \to V (T_2)\).
\end{definition}

\begin{definition}\label{d1.9_from7}
Let \(T_i = T_i(r_i, l_i)\) be labeled rooted trees with the roots \(r_i\) and the labeling  \(l_i \colon V (T_i) \to \mathbb{R}^{+},\) \(i = 1, 2\). An isomorphism \(f \colon V (T_1) \to V (T_2)\) of the rooted trees \(T_1(r_1)\) and \(T_2(r_2)\) is an \emph{isomorphism of the labeled rooted trees} \(T_1(r_1, l_1)\) and \(T_2(r_2, l_2)\) if the equality
\begin{equation*}
l_2(f(v)) = l_1(v)
\end{equation*}
holds for every \(v \in V (T_1)\). The labeled rooted trees \(T_1(r_1, l_1)\) and \(T_2(r_2, l_2)\) are isomorphic if there is an isomorphism of these trees.
\end{definition}

We will say that a tree \(T\) is a \emph{star} if there is a vertex \(c \in V(T)\), the center of \(T\), such that \(c\) and \(v\) are adjacent for every \(v \in V(T) \setminus \{c\}\).

\begin{proposition}\label{p2.2}
A finite connected graph \(G\) with \(\card V(G) = n\) is a tree if and only if \(\card E(G) = n-1\).
\end{proposition}

For the proof see, for example, Corollary~1.5.3 \cite{Diestel2017}.

%The following notion of complete multipartite graph is well-known when the vertex set of the graph is finite (see, for example, \cite[p.~17]{Diestel2017}). Below we need this concept for graphs having vertex sets of arbitrary cardinality.
The next simple proposition directly follows from Definition~\ref{d2.1} and the definition of subgraphs of a graph.

\begin{proposition}\label{p2.3}
Let \(T\) be a tree and let \(G\) be a connected subgraph of \(T\). Then \(G\) is a subtree of \(T\).
\end{proposition}

Let \(\{G_i \colon i \in I\}\) be a nonempty family of graphs such that
\[
\left(\bigcup_{i \in I} V(G_i)\right) \cap \left(\bigcup_{i \in I} E(G_i)\right) = \varnothing.
\]
Then the union \(\bigcup_{i \in I} G_i\) is a graph \(H\) with
\[
V(H) = \bigcup_{i \in I} V(G_i), \quad E(H) = \bigcup_{i \in I} E(G_i).
\]

The definition of connectedness of graphs implies the following.

\begin{proposition}\label{p2.4}
Let \(\{G_i \colon i \in I\}\) be a nonvoid family of connected subgraphs of a graph \(G\). If the set \(\bigcap_{i \in I} V(G_i)\) is nonempty, then \(\bigcup_{i \in I} G_i\) is a connected subgraph of \(G\).
\end{proposition}

In the proof of Theorem~\ref{t4.5} we will use also the next simple fact.

\begin{proposition}\label{p2.5}
Let \( T_1, \, T_2, \, T_3 \ldots\) be a sequence of trees satisfying the inclusion
\begin{equation}\label{e2.5_1}
T_i \subseteq T_{i+1}
\end{equation}
for every integer $i \geqslant 1$. Then the graph
\begin{equation}\label{e2.5_2}
T \colon= \bigcup\limits_{i=1}^{\infty} T_i
\end{equation}
is a tree.

\begin{proof}
Indeed, $T$ is a connected graph by Proposition~\ref{p2.4}. Suppose that we can find a cycle $C \subseteq T$. Since $C$ is a finite graph, inclusion \eqref{e2.5_1} and equality \eqref{e2.5_2} imply that there is an integer $i_0 \geqslant 1 $ such that
$$
T_{i_0} \supseteq C.
$$
The last inclusion is impossible, since $T_{i_0}$ is a tree. Thus $T$ also is a tree.
\end{proof}

\end{proposition}

\begin{definition}\label{d2.5}
Let \(G\) be a graph and let \(k \geqslant 2\) be a cardinal number. The graph \(G\) is \emph{complete \(k\)-partite} if the vertex set \(V(G)\) can be partitioned into \(k\) nonempty, disjoint subsets, or \emph{parts}, in such a way that no edge has both ends in the same part and any two vertices in different parts are adjacent.
\end{definition}

We shall say that $G$ is a \emph{complete multipartite graph} if there is a cardinal number \(k\) such that $G$ is complete $k$-partite.

\begin{lemma}\label{l2.6}
Let \(G\) be a complete multipartite graph. Then the following conditions are equivalent:
\begin{enumerate}
\item\label{l2.6:s1} There is a star \(S \subseteq G\) such that \(V(S) = V(G)\).
\item\label{l2.6:s2} At least one part of \(G\) contains exactly one point.
\end{enumerate}
\end{lemma}

\begin{proof}
\(\ref{l2.6:s1} \Rightarrow \ref{l2.6:s2}\). Let \(S \subseteq G\) be a star with the center \(c\) and let \(V(S) = V(G)\). Then there is a part \(A\) of \(G\) such that \(c \in A\). If \(u\) is a point of \(A\) and \(u \neq c\), then, by Definition~\ref{d2.5}, the points \(u\) and \(c\) are nonadjacent in \(G\). Now \(S \subseteq G\) implies that these points are also nonadjacent in \(S\), contrary to the definition of stars. Thus, the part \(A\) contains the point \(c\) only.

\(\ref{l2.6:s2} \Rightarrow \ref{l2.6:s1}\). Let \(A\) be a part of \(G\) and let \(\card A = 1\) hold. Write \(c\) for the unique point of \(A\) and consider the star \(S\) with the center \(c\) and \(V(S) = V(G)\). Then \(S \subseteq G\) follows from Definition~\ref{d2.5}.
\end{proof}

\section{Preliminaries. Balls and centered spheres in ultrametric spaces}

Let \((X, \rho)\) be a metric space and let \(A\) be a subset of \(X\). Recall that the \emph{diameter} of \(A\) is the quantity
\begin{equation}\label{e3.1}
\diam A = \sup\{\rho(x, y)\colon x, y \in A\}.
\end{equation}

\begin{definition}\label{d3.1}
If \((X, \rho)\) is a metric space with \(\card X \geqslant 2\), then the \emph{diametrical graph} of \((X, \rho)\) is a graph \(G = G_{X, \rho}\) such that \(V(G) = X\) holds and
\begin{equation*}%\label{e3.2}
\bigl(\{u, v\} \in E(G)\bigr) \Leftrightarrow \bigl(\rho(u, v) = \diam X\bigr)
\end{equation*}
is valid for all \(u\), \(v \in V(G)\).
\end{definition}

The following theorem directly follows from Theorem~3.1 \cite{DDP2011pNUAA}.

\begin{theorem}\label{t3.1}
Let \((X, \rho)\) be an ultrametric space with \(\card X \geqslant 2\). If the diametrical graph \(G_{X, \rho}\) is nonempty, then it is a complete multipartite graph.
\end{theorem}

The next lemma shows that the radius of any centered ultrametric sphere is equal to its diameter.

\begin{lemma}\label{l3.1}
Let \(C\) be a centered sphere in an ultrametric space \((X, d)\) and let \(\card C \geqslant 2\). If \(c \in C\) and \(r \in \RR^{+}\) satisfy the condition
\begin{equation}\label{l3.1:e1}
C = \{x \in X \colon d(x, c) = r\} \cup \{c\},
\end{equation}
then the equality
\begin{equation}\label{l3.1:e2}
r = \diam C
\end{equation}
holds.
\end{lemma}

\begin{proof}
The inequality \(\card C \geqslant 2\) implies that there is a point \(x \in C\) such that \(d(x, c) = r\). Consequently,
\begin{equation}\label{l3.1:e3}
r \leqslant \diam C
\end{equation}
holds. Now using \eqref{l3.1:e1} and the strong triangle inequality we obtain
\begin{equation}\label{l3.1:e4}
d(u, v) \leqslant \max \{d(u, c), d(v, c)\} \leqslant r
\end{equation}
for all \(u\), \(v \in C\). Equality \eqref{l3.1:e2} follows from \eqref{l3.1:e3} and \eqref{l3.1:e4}.
\end{proof}

Lemma~\ref{l2.6}, Lemma~\ref{l3.1} and Theorem~\ref{t3.1} give us the following.

\begin{corollary}\label{c3.3}
Let \((Y, \rho)\) be an ultrametric space with nonempty diametrical graph \(G_{Y, \rho}\). Then the following statements are equivalent:
\begin{enumerate}
\item \label{c3.3:s1} \(Y \in \mathbf{Cs}_{Y, \rho}\).
\item \label{c3.3:s2} At least one part of the complete multipartite graph \(G_{Y, \rho}\) contains exactly one point.
\item \label{c3.3:s3} There is a star \(S \subseteq G_{Y, \rho}\) such that \(V(S) = V(G_{Y, \rho})\).
\end{enumerate}
\end{corollary}

The next result is a special case of Proposition~3.3 \cite{BDK2022TaAoG}.

\begin{lemma}\label{l3.3}
Let \((X, \rho)\) be a metric space with \(\card X \geqslant 2\). If the diametrical graph \(G_{X, \rho}\) is complete multipartite, then every part of \(G_{X, \rho}\) is an open ball with a center \(c \in X\) and the radius \(r = \diam X\) and, conversely, every \(B_r(c) \in \mathbf{B}_X\) with \(r = \diam X\) is a part of \(G_{X, \rho}\).
\end{lemma}

Using the last lemma we obtain a refinement of Theorem~\ref{t3.1}.

\begin{theorem}\label{t3.3}
Let \((X, \rho)\) be an ultrametric space with \(\card X \geqslant 2\). If the diametrical graph \(G_{X, \rho}\) is nonempty, then \(G_{X, \rho}\) is complete multipartite and, moreover, the set of all parts of \(G_{X, \rho}\) is the same as the set of all open balls of radius \(r = \diam X\).
\end{theorem}

The following proposition claims that every point of an arbitrary ultrametric ball is a center of that ball.

\begin{proposition}\label{p3.4}
Let \((X, d)\) be an ultrametric space. Then for every ball \(B_r(c)\) and every \(a \in B_r(c)\) we have \(B_r(c) = B_r(a)\).
\end{proposition}

It directly follows from Proposition~18.4~\cite{Sch1985}, so we omit the proof here.

\begin{corollary}\label{c2.4}
Let \((X, d)\) be an ultrametric space. Then the inclusion
\begin{equation*}%\label{c2.4:e1}
\BB_{B, d|_{B \times B}} \subseteq \BB_{X, d}
\end{equation*}
holds for every \(B \in \BB_{X}\)
\end{corollary}

As in the case of Corollary~\ref{c2.4}, Proposition~\ref{p3.4} implies the following.

\begin{corollary}\label{c3.10}
Let \((X, d)\) be an ultrametric space and let \(B \in \mathbf{B}_{X, d}\). Then the inclusion
\begin{equation*}%\label{c3.10:e1}
\mathbf{Cs}_{B, d|_{B \times B}} \subseteq \mathbf{Cs}_{X, d}
\end{equation*}
holds.
\end{corollary}

The following proposition describes some useful properties of locally finite ultrametric spaces.

\begin{proposition}\label{c3.11}
Let \((X, d)\) be a locally finite ultrametric space, $c \in X$, and let \( \mathbf{B}^c_{X, d} \) be the set of all open balls containing the point $c$,
\begin{equation*}
\mathbf{B}^c_{X,d} = \{ B\in \mathbf{B}_{X,d} \, \colon c\in B \}.
\end{equation*}
The following statements hold:
\begin{enumerate}
\item \label{c3.11:s1} The mapping
\begin{equation}\label{c3.11:e1}
\mathbf{B}^c_{X,d}\ni B \mapsto \diam B \in \mathbb{R}^{+}
\end{equation}
is injective.
\item \label{c3.11:s2} If $X$ is infinite, then there is a sequence $( B_1, \, B_2, \, \ldots, B_n, \ldots )$ of balls such that
\begin{equation} \label{c3.11:e2}
\mathbf{B}^c_{X,d} = \{ B_1, \, B_2, \, \ldots, B_n, \ldots \},
\end{equation}
and
\begin{equation}\label{c3.11:e3}
\lim\limits_{n\to \infty} \diam B_n = \infty ,
\end{equation}
and
\begin{equation}\label{c3.11:e4}
\diam B_n < \diam B_{n+1}
\end{equation}
for every positive integer $n$.
\end{enumerate}
\end{proposition}

\begin{proof}

$(i)$. Since $(X,d)$ is locally finite, every $B \in \mathbf{B}^c_{X,d}$ can be represented as
    $$
    B = \{x\in X \colon d(x,c) \leqslant \diam B\},
    $$
    that implies the injectivity of mapping \eqref{c3.11:e1}.

$(ii)$. Let $X$ be infinite. Since \( (X,d) \) is locally finite, the set
\begin{equation*}
    \{ d(c,x) \colon x \in X \} \cap [0,t]
\end{equation*}
is finite for every $t \in \mathbb{R}^{+}$. Moreover, the set
    \begin{equation*}
    D^{c}_{1} = \{ d(c,x) \colon x \in X \}
\end{equation*}
is unbounded because every bounded locally finite metric space is finite. Using the last two assertions, it is easy to check that the sets \( D^{c}_{1} \) and
$$
\mathbb{N} = \{ 1, 2, \ldots, n, \ldots\}
$$
are order-isomorphic as subsets of the order set \( (\mathbb{R}^+, \leqslant )\). Let \( \Phi : \mathbb{N} \to D^{c}_{1} \) be an order-isomorphism of $\mathbb{N}$ and $D_1^c$. Write
\begin{equation*}
    t_n \colon= \Phi (n)
\end{equation*}
for every $n \in \mathbb{N}$. Then, by definition of the order-isomorphisms we have
\begin{equation}\label{c3.11:e7}
t_n < t_{n+1}
\end{equation}
for each $n \in \mathbb{N}$. Furthermore, the limit relation
\begin{equation}\label{c3.11:e8}
\lim\limits_{n \to \infty} t_n = \infty
\end{equation}
holds, since $D_{1}^c$ is an unbounded subset of $\mathbb{R}^+$.

Let us now denote by $D^c_2$ the set $ \{ \diam B \colon B \in \mathbf{B}^c_{X,d}\} $. We claim that the equality
\begin{equation}\label{c3.11:e9}
D_1^c = D_2^c
\end{equation}
holds.

Indeed, since \( (X,d)\) is locally finite, in each $B \in \mathbf{B}^c_{X,d}$ we can find $p \in B$ satisfying  the equality
$$
d(c,p) = \diam B.
$$
Consequently, the inclusion
\begin{equation}\label{c3.11:e10}
D_1^c \subseteq D_2^c
\end{equation}
holds. Now, again using the local finiteness of $(X,d)$ for each $a \in X$ we can find $\varepsilon>0$, such that the set
$$
\{ x \in X \colon d(c,a) < d(c,x)<d(c,a) +\varepsilon \}
$$
is empty, that implies the equality
$$
\diam B_r(c)=d(c,a),
$$
whenever $r \in (d(c,a), d(c,a)+\varepsilon)$. Thus the inclusion
\begin{equation}\label{c3.11:e11}
D_2^c \supseteq D_1^c
\end{equation}
holds.  Equality\eqref{c3.11:e9} follows from \eqref{c3.11:e10} and \eqref{c3.11:e11}.

Statement $(i)$ implies that there is a bijection $F \colon D_2^c \to \mathbf{B}^c_{X,d}$ satisfying the equality
$$
\diam F(t)=t
$$
for every $t \in D_2^c$.

Let us consider now the bijective mapping
$$
\mathbb{N} \xrightarrow{\Phi} D_1^c \xrightarrow{\mathrm{Id}} D_2^c  \xrightarrow{F} \mathbf{B}^c_{X,d}
$$
where $\mathrm{Id} \colon D_1^c \to D_2^c$ is the identical mapping, and define $B_n \in \mathbf{B}^c_{X,d}$ as
%%%%%%%%%%%%%%%%%%%%%%%%%%%%%%%%%%%%%%%%%%%%%%%%%%%%5
the value of this mapping at point $n \in \mathbb{N}$.
%%%%%%%%%%%%%%%%%%%%%%%%%%%%%%%%%%%%%%%%%%%%%%%%
%$$
%B_n \colon = \Phi \circ \mathrm{Id} \circ F(n).
%$$
Then \eqref{c3.11:e3} and \eqref{c3.11:e4} follows from \eqref{c3.11:e7} and \eqref{c3.11:e8} respectively.

\end{proof}

\section{Characterization of locally finite \(\mathbf{UGVL}\)-spaces}

First of all we note that class \(\mathbf{UGVL}\) contains all nonempty ultrametric spaces with at most \(3\) points.

\begin{proposition}\label{p4.1}
Let \((X, d)\) be a nonempty ultrametric space. If the inequality \(\card X \leqslant 3\) holds, then \((X, d) \in \mathbf{UGVL}\) and every \(B \in \mathbf{B}_X\) is a centered sphere in \((X, d)\).
\end{proposition}

\begin{proof}
If \(\card X = 1\) or \(\card X = 2\), then the proposition is trivially valid. Let us consider the case when \(\card X = 3\), \(X = \{x, y, z\}\).

Every triangle in any ultrametric space is isosceles, and its base has the length less than or equal to that of its legs. Thus, we may suppose that
\[
d(x, y) = d(y, z) = a \text{ and } d(z, x) = b
\]
with \(a \geqslant b > 0\). Let us consider now a labeled path \(P_1 = P_1(l)\) with
\[
V(P_1) = \{y, x, z\} \text{ and } E(P_1) = \{\{y, x\}, \{x, z\}\}
\]
and labeling \(l \colon V(P_1) \to \mathbb{R}^+\) such that
\[
l(y) = a, \quad l(x) = 0\quad \text{and}\quad l(z) = b
\]
(see Figure~\ref{fig1}). Then \(P_1\) is a labeled tree. A simple calculation shows that \(d = d_l\) holds, where \(d_l\) is defined by \eqref{e1.1} with \(T = P_1\). Thus, \((X, d)\) belongs to the class \(\mathbf{UGVL}\) by the definition.

Let us prove that every \(B \in \mathbf{B}_X\) is a centered sphere in \((X, d)\). The last statement holds if \(\card B = 1\), that follows from \eqref{d1.2:e1} with \(S = B\) and \(r = 0\).

If \(\card B = 2\) or \(\card B = 3\), then to see that \(B\) is a centered sphere one can use Corollary~\ref{c3.3} with \((Y, \rho) = (B, d|_{B \times B})\).
\end{proof}

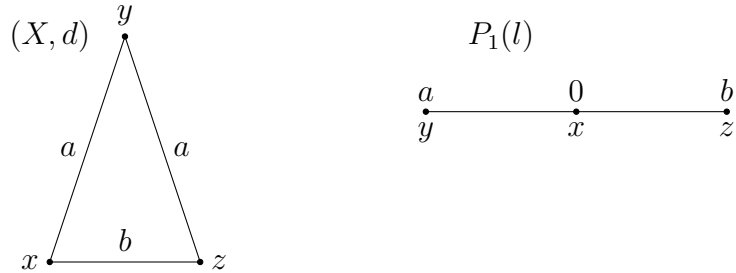
\begin{figure}[htb]
\begin{tikzpicture}
\coordinate [label=left: $x$] (x) at (0,0);
\coordinate [label=above: $y$] (y) at (1,3);
\coordinate [label=right: $z$] (z) at (2,0);
\draw (x) -- node [left] {\(a\)} (y) -- node [right] {\(a\)} (z) -- node [above] {\(b\)} (x);

\node at (0, 3) {\((X, d)\)};

\draw [fill, black] (x) circle (1pt);
\draw [fill, black] (y) circle (1pt);
\draw [fill, black] (z) circle (1pt);

\coordinate [label=above: $a$] (a) at (5,2);
\coordinate [label=above: $0$] (O) at (7,2);
\coordinate [label=above: $b$] (b) at (9,2);
\node at (6, 3) {\(P_1(l)\)};

\draw [fill, black] (a) circle (1pt);
\draw [fill, black] (O) circle (1pt);
\draw [fill, black] (b) circle (1pt);

\draw (a) node[anchor=north] {$y$} -- (O) node[anchor=north] {$x$} -- (b) node[anchor=north] {$z$};
\end{tikzpicture}
\caption{The ultrametric triangle \((X, d)\) is generated by the labeled path \(P_1(l)\).}
\label{fig1}
\end{figure}

The following example shows that \(3\) is the best possible constant in Proposition~\ref{p4.1}.

\begin{example}\label{ex4.2}
Let us consider the four-point ultrametric space \((X, d)\) depicted by Figure~\ref{fig2}. To see that there is no labeled tree for which
\begin{equation}\label{ex4.2:e1}
d_l = d
\end{equation}
holds, suppose that, for some tree \(T\) with \(V(T) = \{x, y, z, t\}\) and \(l \colon V(T) \to \mathbb{R}^{+}\), \eqref{ex4.2:e1} holds. Then, using \eqref{e1.1}, we obtain
\begin{align*}
d_l(x, z) & = 1 = \max \{l(x), l(z)\},\\
d_l(y, t) & = 1 = \max \{l(y), l(t)\}.
\end{align*}
That implies
\[
\diam X = \max \{l(x), l(y), l(z), l(t)\} = 1,
\]
contrary to \(\diam X \geqslant d(x, y) = 2\).
\end{example}

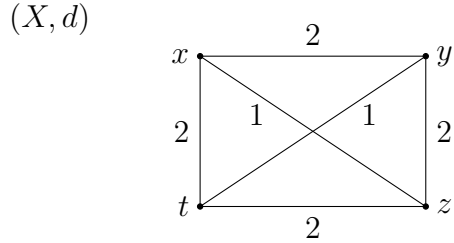
\begin{figure}[htb]
\begin{tikzpicture}
\coordinate [label=left: $x$] (x) at (0,2);
\coordinate [label=right: $y$] (y) at (3,2);
\coordinate [label=right: $z$] (z) at (3,0);
\coordinate [label=left: $t$] (t) at (0,0);
\draw (x) -- node [above] {\(2\)} (y) -- node [right] {\(2\)} (z) -- node [below] {\(2\)} (t) -- node [left] {\(2\)} (x);

\draw (x) -- node [near start, below] {\(1\)} (z);
\draw (y) -- node [near start, below] {\(1\)} (t);

\node at (-2, 2.5) {\((X, d)\)};

\draw [fill, black] (x) circle (1pt);
\draw [fill, black] (y) circle (1pt);
\draw [fill, black] (z) circle (1pt);
\draw [fill, black] (t) circle (1pt);
\end{tikzpicture}
\caption{The four-point ultrametric space \((X, d)\) does not belong to \(\mathbf{UGVL}\).}
\label{fig2}
\end{figure}

Let us show that every open ball in an \(\mathbf{UGVL}\)-space is also an \(\mathbf{UGVL}\)-space.

\begin{lemma}\label{l4.3}
Let \((X, d) \in \mathbf{UGVL}\) and let \(T(l)\) be a labeled tree generating \((X, d)\). Then, for every \(B^1 \in \mathbf{B}_X\), there is a subtree \(T^1\) of \(T\) such that
\begin{equation}\label{l4.3:e1}
V(T^1) = B^1
\end{equation}
and
\begin{equation}\label{l4.3:e2}
d|_{B^1 \times B^1} = d_{l^1}
\end{equation}
holds, where \(l^1\) is the restriction of the labeling \(l \colon V(T) \to \mathbb{R}^+\) on the set \(V(T^1)\).
\end{lemma}

\begin{proof}
Let \(B^1 = B_{r_1}(c_1)\) be an arbitrary open ball in \((X, d)\). If \(\card B^1 = 1\) holds, then the empty tree \(T^1\) with \(V(T^1) = \{c_1\}\) satisfy also \eqref{l4.3:e2}.

Suppose that \(\card B^1 \geqslant 2\) and consider the family
\[
\mathcal{F}_{B^1} = \{P^x \colon x \in B^1, x \neq c_1\},
\]
where \(P^x\) is the unique path joining \(c_1\) and \(x\) in \(T\). Then, by Proposition~\ref{p2.4}, the union
\begin{equation}\label{l4.3:e3}
T^1 := \bigcup_{P^x \in \mathcal{F}_{B^1}} P^x
\end{equation}
is a connected subgraph of \(T\) and, consequently, \(T^1\) is a subtree of \(T\) by Proposition~\ref{p2.3}. It follows directly from \eqref{l4.3:e3} that the inclusion \(V(T^1) \supseteq B^1\) holds. Thus, to prove equality \eqref{l4.3:e1} it suffices to show that the inclusion
\begin{equation}\label{l4.3:e4}
V(P^x) \subseteq B^1
\end{equation}
is valid for every \(P^x \in \mathcal{F}_{B^1}\).

Let us consider an arbitrary \(P^x \in \mathcal{F}_{B^1}\),
\begin{align*}
V(P^x) &= \{x_0, x_1, \ldots, x_k\},\\
E(P^x) &= \bigl\{\{x_0, x_1\}, \{x_1, x_2\}, \ldots, \{x_{k-1}, x_k\}\bigr\}, \quad k \geqslant 1,
\end{align*}
\(x_0 = c_1\) and \(x_k = x\). Then, using \eqref{e1.1}, we obtain
\begin{align*}
d(x_0, x_j) & = d_l(x_0, x_j) = \max_{1 \leqslant i \leqslant j} l(v_i) \\
& \leqslant \max_{1 \leqslant i \leqslant k} l(v_i) = d(x_0, x) < r_1
\end{align*}
for every \(j \in \{1, \ldots, k\}\). Thus,
\begin{equation}\label{l4.3:e5}
x_j \in B^1
\end{equation}
holds for every \(j \in \{1, \ldots, k\}\). Now \(x_0 = c_1\), \(c_1 \in B\) and \eqref{l4.3:e5} imply \eqref{l4.3:e4}.

To complete the proof it suffices to note that \eqref{l4.3:e2} follows from \eqref{e1.1}, since we have \(d = d_l\) and, for every pair of distinct \(u\), \(v \in V(T^1)\), there is the unique path \(P\) joining \(u\) and \(v\) in \(T\), and that \(P \subseteq T^1\) (because \(T^1\) is a subtree of \(T\)).
\end{proof}

 The next theorem can be proved using the Gurvich--Vyalyi representing tree technique (see Theorem~4.1 in \cite{Pet}) but we will  give an independent proof that allows us to obtain a similar result for locally finite spaces.

\begin{theorem}\label{t4.4}
The statements
\begin{enumerate}
\item\label{t4.4:s1} \((X, d) \in \mathbf{UGVL}\)

\noindent and
\item\label{t4.4:s2} \(\mathbf{B}_{X, d} \subseteq \mathbf{Cs}_{X, d}\)
\end{enumerate}
are equivalent for every finite nonempty ultrametric space \((X, d)\).
\end{theorem}

\begin{proof}
\(\ref{t4.4:s1} \Rightarrow \ref{t4.4:s2}\). By Proposition~\ref{p4.1}, the logical equivalence \(\ref{t4.4:s1} \Leftrightarrow \ref{t4.4:s2}\) is valid if \(\card X \leqslant 3\) holds. Thus, without loss of generality, we can assume that
\begin{equation}\label{t4.4:e1}
\card X \geqslant 4.
\end{equation}

Let \((X, d)\) belong to the class \(\mathbf{UGVL}\). Then there is a labeled tree \(T = T(l)\) such that \(V(T) = X\) and \(d_l = d\) holds. We must show that the inclusion
\begin{equation}\label{t4.4:e4}
\mathbf{Cs}_{X, d} \supseteq \mathbf{B}_{X, d}
\end{equation}
is valid, i.e., every open ball \(B\) in \((X, d)\) is a centered sphere in \((X, d)\). Let us make sure that the last statement is true for the case \(B = X\).

The finiteness of \(X\) and inequality \eqref{t4.4:e1} imply that the diametrical graph \(G_{X, d}\) is nonempty. Using Corollary~\ref{c3.3}, we obtain that \(X \in \mathbf{Cs}_{X, d}\) holds if and only if at least one part of the complete multipartite graph \(G_{X, d}\) contains exactly one point. Let \(\{A_1, \ldots, A_k\}\) be the set of all parts of \(G_{X, d}\). Suppose contrary that the inequality
\begin{equation}\label{t4.4:e5}
\card A_i \geqslant 2
\end{equation}
holds for every \(i \in \{1, \ldots, k\}\). Let us consider a subset \(\{c_1, \ldots, c_k\}\) of the set \(X\) such that \(c_i \in A_i\) for every \(i \in \{1, \ldots, k\}\). Then, by Theorem~\ref{t3.3} and Proposition~\ref{p3.4}, for every \(i \in \{1, \ldots, k\}\) we have
\begin{equation}\label{t4.4:e6}
A_i = B_r(c_i)
\end{equation}
with \(r = \diam X\). Lemma~\ref{l4.3} implies now that all ultrametric spaces \((A_1, d|_{A_1 \times A_1})\), \ldots, \((A_k, d|_{A_k \times A_k})\) belong to the class \(\mathbf{UGVL}\). In particular, by Lemma~\ref{l4.3}, there are labeled subtrees \(T^1(l_1)\), \ldots, \(T^k(l_k)\) of the labeled tree \(T(l)\) such that
\begin{equation}\label{t4.4:e7}
V(T^i) = A_i \quad \text{and} \quad d|_{A_i \times A_i} = d_{l_i}
\end{equation}
holds with \(l_i = l|_{A_i}\) for every \(i \in \{1, \ldots, k\}\). Now using formula~\eqref{t4.4:e6} with \(r = \diam X\) and \eqref{t4.4:e7} we obtain the strict inequality
\begin{equation}\label{t4.4:e8}
\max_{u \in A_i} l(u) < \diam X
\end{equation}
for every \(i \in \{1, \ldots, k\}\). Since the number \(k\) of parts of \(G_{X, d}\) is finite and \(\{A_1, \ldots, A_k\}\) is a partition of \(X\), inequality \eqref{t4.4:e8} give us
\begin{equation}\label{t4.4:e9}
\max_{u \in X} l(u) = \max_{1 \leqslant i \leqslant k} \max_{u \in A_i} l(u) < \diam X.
\end{equation}
Now to complete the proof of validity \(\ref{t4.4:s1} \Rightarrow \ref{t4.4:s2}\) it suffices to note that the finiteness of \(X\) and the definition of the ultrametric \(d_l\) imply the equality
\[
\max_{u \in X} l(u) = \diam X,
\]
contrary to \eqref{t4.4:e9}.

\(\ref{t4.4:s2} \Rightarrow \ref{t4.4:s1}\). We must show that
\begin{equation}\label{t4.4:e10}
(X, d) \in \mathbf{UGVL}
\end{equation}
whenever \((X, d)\) is a finite nonempty ultrametric space satisfying the inclusion
\begin{equation}\label{t4.4:e11}
\mathbf{B}_{X, d} \subseteq \mathbf{Cs}_{X, d}.
\end{equation}
To prove the above statement we will use the induction on \(\card X\).

By Proposition~\ref{p4.1} we obtain that \(\eqref{t4.4:e11} \Rightarrow \eqref{t4.4:e10}\) is valid for every ultrametric space \((X, d)\) with \(1 \leqslant \card X \leqslant 3\).

Let \(n \geqslant 3\) be a given integer number. Suppose that \(\eqref{t4.4:e11} \Rightarrow \eqref{t4.4:e10}\) is valid if
\begin{equation}\label{t4.4:e12}
1 \leqslant \card X \leqslant n.
\end{equation}
Let us consider an arbitrary fixed ultrametric space \((X, d)\) such that \(\card X = n+1\) and \eqref{t4.4:e11} holds.

Let \(\{A_1, \ldots, A_k\}\) be the set of all parts of the diametrical graph \(G_{X, d}\). By Theorem~\ref{t3.3}, every \(A_i\), \(i \in \{1, \ldots, k\}\), is an open ball in \((X, d)\). Now, using Corollaries~\ref{c2.4} and \ref{c3.10}, we see that \eqref{t4.4:e11} implies the inclusion
\begin{equation}\label{t4.4:e13}
\mathbf{B}_{A_i, d|_{A_i \times A_i}} \subseteq \mathbf{Cs}_{A_i, d|_{A_i \times A_i}}
\end{equation}
for every \(i \in \{1, \ldots, k\}\). Since each \(A_i\) is a proper subset of \(X\), the induction hypothesis gives us the membership
\[
\bigl(A_i, d|_{A_i \times A_i}\bigr) \in \mathbf{UGVL}
\]
for every \(i \in \{1, \ldots, k\}\). Consequently, for every \(i \in \{1, \ldots, k\}\), we can find a labeled tree \(T^i(l_i)\) such that
\begin{equation}\label{t4.4:e14}
V(T^i) = A_i \quad \text{and} \quad d|_{A_i \times A_i} = d_{l_i}.
\end{equation}

Let \(\{c_1, \ldots, c_k\}\) be a subset of the set \(X\) such that \(c_i \in A_i\) holds for every \(i \in \{1, \ldots, k\}\). By Corollary~\ref{c3.3}, equality \eqref{t4.4:e11} implies that there is \(i \in \{1, \ldots, k\}\) such that \(\card A_i = 1\). Without loss of generality, we may assume that the set \(A_1\) is a singleton, \(A_1 = \{c_1\}\).

Let us expand the labeled tree \(T^i = T^i(l_i)\) to an labeled tree \(T_1^i = T_1^i(l_{i, 1})\) by the rule:

\(V(T_1^i) = \{c_1\} \cup V(T^i)\), \(E(T_1^i) = \{c_1, c_i\} \cup E(T^i)\) and
\begin{equation}\label{t4.4:e15}
l_{i, 1} = \begin{cases}
l_i(u) & \text{if } u \in V(T^i)\\
\diam X & \text{if } u = c_1
\end{cases}
\end{equation}
for every \(i \in \{2, \ldots, k\}\).

By Proposition~\ref{p2.4}, the graph
\[
T = \bigcup_{i=2}^k T_1^i
\]
is connected. Now, using Proposition~\ref{p2.2}, we can prove that \(T\) is a tree. Indeed, by Proposition~\ref{p2.2}, \(T\) is a tree iff
\begin{equation}\label{t4.4:e2}
\card V(T) - \card E(T) = 1.
\end{equation}
To prove the last equality we note that
\begin{align*}
\card V(T) &= \sum_{i=1}^{k} \card A_i = 1 + \sum_{i=2}^{k} \card V(T^i) \\
& = 1 + \sum_{i=2}^{k} \bigl(\card V(T_1^i) - 1\bigr) = 2 - k + \sum_{i=2}^{k} \card V(T_1^i)\\
\intertext{and}
\card E(T) &= \sum_{i=2}^{k} \card E(T_1^i).
\end{align*}
Consequently, we have
\begin{equation}\label{t4.4:e2.1}
\card V(T) - \card E(T) = 2 - k + \sum_{i=2}^{k} \bigl(\card V(T_1^i) - \card E(T_1^i)\bigr).
\end{equation}
Since every \(T_1^i\) is a tree, \(\card V(T_1^i) - \card E(T_1^i) = 1\) holds for each \(i \in \{2, \ldots, k\}\). Thus, the right half of formula~\eqref{t4.4:e2.1} can be written as
\[
2 - k + \sum_{i=2}^{k} \bigl(\card V(T_1^i) - \card E(T_1^i)\bigr) = 2 - k + (k-1) = 1,
\]
that implies \eqref{t4.4:e2}.

Using \eqref{t4.4:e15} we can find a labeling \(l \colon V(T) \to \mathbb{R}^+\) such that
\begin{equation}\label{t4.4:e16}
l|_{V(T_1^i)} = l_{i,1}
\end{equation}
holds for every \(i \in \{2, \ldots, k\}\). Then we have \(V(T) = X\) and, in addition, equalities \eqref{t4.4:e14}, \eqref{t4.4:e15} and \eqref{t4.4:e16} imply the equality \(d_l = d\). Thus, \eqref{t4.4:e10} is valid.
\end{proof}

The second part of the proof of Theorem~\ref{t4.4} (see, in particular, formula~\eqref{t4.4:e16}) gives us the following.

\begin{corollary}\label{c4.4}
Let $(Y, \rho) \in \mathbf{UGVL}$ be finite, let the diametrical graph $G_{Y, \rho}$ be complete multipartite with parts $B_1^Y, \ldots, B_n^Y$ and let \linebreak $T_1=T_1(l), \ldots, T_n=T_n(l)$ be labeled trees generating, respectively, the ultrametric spaces $\left(B_1^Y, \, \rho_{|B_1^Y \times B_1^Y}\right), \ldots, \left(B_n^Y, \, \rho_{|B_n^Y \times B_n^Y}\right)$. Then there exists a labeled tree $T=T(l)$ generating $(Y, \rho)$ such that $T_i \subseteq T$ and $ l_{|V(T_i)} = l_i $ for every $i=1, \, 2\, \ldots, n$.

\end{corollary}

Let us now turn to the case of locally finite ultrametric spaces.
\medskip

The following theorem is the first main result of the paper.

\begin{theorem}\label{t4.5}
Let \((X, d)\) be a locally finite nonempty ultrametric space. Then the following statements are equivalent:
\begin{enumerate}
\item \((X, d) \in \mathbf{UGVL}\).
\item \(\mathbf{B}_{X, d} \subseteq \mathbf{Cs}_{X, d}\).
\item For every \(B \in \mathbf{B}_{X, d}\) with \(\card B \geqslant 2\) there is a star $S$ such that $V(S)=B$ and
$S \subseteq G_{B, d_{|B \times B}}$ where $G_{B, d_{|B \times B}}$ is the diametrical graph of the space $(B, d_{|B \times B})$.
\end{enumerate}
\end{theorem}

\begin{proof}
Corollaries~\ref{p3.4}, \ref{c2.4} and Corollary~\ref{c3.3} show that the logical equivalence $(ii) \Leftrightarrow (iii)$ is valid.

Moreover, if $(X,d) \in \mathbf{UGVL}$ holds, then, for every $B \in \mathbf{B}_{X,d}$ we have $ (B, d_{|B\times B}) \in \mathbf{UGVL}$ by Lemma~\ref{l4.3}. Consequently, using Corollary~\ref{p3.4}, Corollary~\ref{c2.4}, and the finiteness of balls in locally finite metric spaces we see that the validity of $(i) \Rightarrow (ii)$ follows from Theorem~\ref{t4.4}.

To complete the proof it suffices to show that $(ii) \Rightarrow (i)$ is also valid.

Let us consider the case when $(X,d)$ is infinite. In the case when $(X,d)$ is finite the validity of $(ii) \Rightarrow (i)$ was proved in Theorem~\ref{t4.4}.

Suppose that $(ii)$ holds. Let $c$ be a point of $X$ and let
$$
\mathbf{B}^c_{X,d} \colon= \{B \in \mathbf{B}_{X,d} \colon c\in B \}.
$$
Then, by Proposition~\ref{c3.11} there exists an infinite sequence $(B_n)_{n \in \mathbb{N}}$ of open balls satisfying the conditions:
\begin{itemize}
    \item[$(s_1)$] \label{t4.5:s2} $\diam B_n < \diam B_{n+1}$ for every $n \in \mathbb{N}$;
    \item[$(s_2)$] \label{t4.5:s3} $\mathbf{B}^c_{X,d} = \{B_n \colon n \in \mathbb{N} \}$.
 \end{itemize}

Every open ball is finite subset of $X$ because $(X,d)$ is locally finite. Consequently by Theorem~\ref{t4.4}, the ultrametric space $\left( B_n, d_{|B_n \times B_n} \right)$ belongs to the class $\mathbf{UGVL}$ for every $n \in \mathbb{N}$.

Now using Corollary~\ref{c4.4} and statements $(s_1)$ we can find a sequence $(T_n)_{n\in \mathbb{N}}$ of labeled trees $T_n=T_n(l_n)$ such that:
\begin{itemize}
    \item[$(s_3)$] \label{t4.5:s4} $\left( B_n, d_{|B_n \times B_n} \right)$ is generated by $T_n(l_n)$;
    \item[$(s_4)$] \label{t4.5:s5} $T_n \subseteq T_{n+1}$  and $l_{n+1 | V(T_n)} = l_n$ for every $n \in \mathbb{N}.$
\end{itemize}
Write
\begin{equation*}
    T \colon = \bigcup\limits_{n=1}^{\infty} T_n.
\end{equation*}
Then $T$ is a tree by Proposition~\ref{p2.5}. It follows from statements ($s_4$)
%\ref{t4.5:s5}
and the equality
\begin{equation} \label{t4.5_p:e1}
    V(T) = \bigcup\limits_{n=1}^{\infty} V(T_n)
\end{equation}
that there is a labeling $l \colon V(T) \to \mathbb{R}^+$ such that
\begin{equation}\label{t4.5_p:e2}
    l_{|V(T)} = l_n
\end{equation}
for every $n \in \mathbb{N}$. Statements $(s_2)$ and $(s_3)$ give us
$$
X= \bigcup\limits_{n=1}^{\infty} B_n= \bigcup\limits_{n=1}^{\infty} V(T_n),
$$
which together with \eqref{t4.5_p:e1} implies the equality
$$V(T)=X.$$

Now using the last equality, equality~\eqref{t4.5_p:e2} and $(s_3)$--$(s_4)$ it is easy to show that $(X,d)$ is generated by labeled tree $T=T(l)$.

Thus the membership $(X,d) \in \mathbf{UGVL}$ is valid.
\end{proof}

Theorem \ref{t4.5} claims, in particular, that the inclusion \( \mathbf{B}_{X,d}\subseteq \mathbf{Cs}_{X,d} \) implies \( (X,d) \in \mathbf{UGVL}\) for locally finite ultrametric spaces \( (X, d)\). In the rest of the section, we want to show that the reserve inclusion \(\mathbf{Cs}_{X,d} \subseteq \mathbf{B}_{X,d}\) holds iff the metric \(d\) is discrete.

We say that a metric \( d \colon X \times X \to \mathbb{R}^+\) is \emph{discrete} if there is a constant $k>0$ such that
\begin{equation}\label{4.2.3}
d(x,y)=k
\end{equation}
whenever $x$ and $y$ are distinct points of $X$.

\begin{remark}\label{rem_1.9_whrn_all}
The standard definition of \emph{discrete metric} can be formulated as:
``The metric on X is discrete if the distance from each point of X to every
other point of X is one.'' (See, for example, \cite[p.~4]{Sea2007}.)
\end{remark}

\begin{lemma}\label{l4.7}
The following conditions are equivalent for every metric space \((X, d)\):
\begin{enumerate}
\item \label{l4.7:i} The metric  \(d\) is discrete.
\item \label{l4.7:ii} For each $x\in X$ there is $k>0$ such that \eqref{4.2.3} holds whenever  \( y \in X \setminus \{x\}\).
\end{enumerate}
\end{lemma}

\begin{proof}
\(\ref{l4.7:i} \Rightarrow \ref{l4.7:ii}\). This implication is trivially valid.

\(\ref{l4.7:ii} \Rightarrow \ref{l4.7:i}\). Let \(\ref{l4.7:ii}\) hold but \( d \) not be a discrete metric. Then there are some points $x, y, u, v \in X$ such that
\begin{equation}\label{l4.7_p_4.24}
d(x,y) \neq d(u,v),
\end{equation}
and
\begin{equation}\label{l4.7_p_4.25}
\min\{ d(x,y), d(u,v)\} >0.
\end{equation}

If the sets \( \{x,y\}\) and \(\{u,v\}\) have a common point, then, without loss of generality we suppose $x=u$. From \eqref{l4.7_p_4.24} and \eqref{l4.7_p_4.25} it follows that
$$ x \neq y, \  u \neq v, \ \textrm{ and } d(x,y) \neq d(u,v),$$
contrary to  \(\ref{l4.7:ii}\). Consequently the sets $\{x,y\}$ and $\{u,v\}$ are disjoint.

Now using condition \(\ref{l4.7:ii}\) again, we obtain
$$
d(x,y)=d(x,u) \neq 0
$$
and
$$
d(u,x)=d(u,v) \neq 0,
$$
that implies $d(x,y)=d(u,v)$. The last equality contradicts  \eqref{l4.7_p_4.24}. The validity of \(\ref{l4.7:ii} \Rightarrow \ref{l4.7:i}\) follows.
\end{proof}

\begin{theorem}\label{t4.8}
Let \((X, d)\) be a nonempty ultrametric space. Then the following statements are equivalent:
\begin{enumerate}
\item \( \mathbf{B}_{X,d} \supseteq \mathbf{Cs}_{X, d} \) .
\item   The metric  \(d\) is discrete.
\item \( \mathbf{B}_{X,d} = \mathbf{Cs}_{X, d} \).
\end{enumerate}
\end{theorem}

\begin{proof}
The implications \( (ii) \Rightarrow (iii) \) and \( (iii) \Rightarrow (i) \) are evidently valid. Let us prove the validity of \( (i) \Rightarrow (ii) \).

Let $(i)$ hold but $d$ not be a discrete metric. Then, by Lemma~\ref{l4.7} there are distinct points $a, b, c \in X$ such that
\begin{equation}\label{e4.26}
d(c,a) > d(c,b) >0.
\end{equation}
Write
\begin{equation}\label{e4.27}
    C \colon = \{ x \in X \colon d(x,c)=r \} \cup \{c\}
\end{equation}
where
\begin{equation}\label{e4.28}
    r= d(c,a).
\end{equation}
Then $C$ is a centered sphere in \( (X,d), a \in C\), and
\begin{equation}\label{e4.29}
   b \notin C,
\end{equation}
by \eqref{e4.26}--\eqref{e4.28}. By condition $(i)$ there is \(B_1 \in \mathbf{B}_{X,d}\),
\begin{equation*}
    B_1 \colon= \{ x \in X \colon d(x,c_1) < r_1\}
\end{equation*}
such that $B_1=C$. Since $c \in C$, Proposition~\ref{p3.4} implies that a center $c$ of the centered sphere $C$ also is the center of the ball $B_1$,
\begin{equation}\label{e4.30}
     B_1 = \{ x \in X \colon d(x,c) < r_1\}.
\end{equation}
Now using \eqref{e4.28} and $a \in C$  we obtain the inequality $r<r_1$. Consequently $b \in B_1$ holds by \eqref{e4.26} and \eqref{e4.30}. To complete the proof it suffices to note that $b \in B_1$ and $B_1=C$ give us $b \in C$, contrary to \eqref{e4.29}.
\end{proof}

\section{Isometric embedding of finite ultrametric spaces in \(\mathbf{UGVL}\)-spaces}

Now we want to show that any finite ultrametric space can be extended to some minimal \(\mathbf{UGVL}\)-space, and that such an extension is unique up to isometry.

\begin{definition}\label{d5.1}
Let \((X, d)\) be a ultrametric space. An \(\mathbf{UGVL}\)-space \((Y, \rho)\) is an \(\mathbf{UGVL}\)-\emph{extension} of \((X, d)\) if there is \(Y_1 \subseteq Y\) such that \((Y_1, \rho|_{Y_1 \times Y_1})\) is isometric to \((X, d)\).
\end{definition}

In what follows we will say that an \(\mathbf{UGVL}\)-extension \((Y, \rho)\) of \((X, d)\) is \emph{minimal} if, for every proper subset \(Y_0\) of \(Y\), the ultrametric space \((Y_0, \rho|_{Y_0 \times Y_0})\) is not an \(\mathbf{UGVL}\)-extension of \((X, d)\).

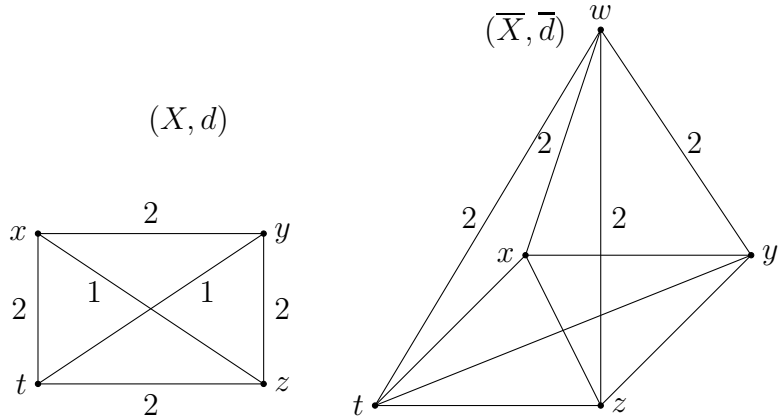
\begin{figure}[htb]
\begin{tikzpicture}[scale=1]
\coordinate [label=left: $x$] (x) at (0,2);
\coordinate [label=right: $y$] (y) at (3,2);
\coordinate [label=right: $z$] (z) at (3,0);
\coordinate [label=left: $t$] (t) at (0,0);
\draw (x) -- node [above] {\(2\)} (y) -- node [right] {\(2\)} (z) -- node [below] {\(2\)} (t) -- node [left] {\(2\)} (x);

\draw (x) -- node [near start, below] {\(1\)} (z);
\draw (y) -- node [near start, below] {\(1\)} (t);

\node at (2, 3.5) {\((X, d)\)};

\draw [fill, black] (x) circle (1pt);
\draw [fill, black] (y) circle (1pt);
\draw [fill, black] (z) circle (1pt);
\draw [fill, black] (t) circle (1pt);
\end{tikzpicture}
\quad
\begin{tikzpicture}[scale=1]
\coordinate [label=left: $x$] (x) at (0,2);
\coordinate [label=right: $y$] (y) at (3,2);
\coordinate [label=right: $z$] (z) at (1,0);
\coordinate [label=left: $t$] (t) at (-2,0);
\coordinate [label=above: $w$] (w) at (1,5);
\draw (x) -- (y) -- (z) -- (t) -- (x);

\draw (x) -- node [left] {\(2\)} (w) -- node [right] {\(2\)} (y);
\draw (t) --  node [left] {\(2\)}  (w) -- node [right] {\(2\)} (z);

\draw (x) -- (z);
\draw (y) -- (t);

\node at (0, 5) {\((\overline{X}, \overline{d})\)};

\draw [fill, black] (x) circle (1pt);
\draw [fill, black] (y) circle (1pt);
\draw [fill, black] (z) circle (1pt);
\draw [fill, black] (t) circle (1pt);
\draw [fill, black] (w) circle (1pt);
\end{tikzpicture}
\caption{The pyramid \((\overline{X}, \overline{d})\) is a minimal $\mathbf{UGVL}$-extension of the quadruple \((X, d)\).}
\label{fig3}
\end{figure}

\begin{example}\label{ex5.2}
Let \((X, d)\), \(X = \{x, y, z, t\}\), be the four-point ultrametric space depicted by Figure~\ref{fig2}. It was shown in Example~\ref{ex4.2} that \((X, d) \notin \mathbf{UGVL}\). Let us consider the five-point set \(\overline{X} = \{x, y, z, t, w\}\), and define an ultrametric \(\overline{d}\) on \(\overline{X}\) such that \(\overline{d}|_{X \times X} = d\) and \(\overline{d}(w, p) = 2\) whenever \(p \in X\) (see Figure~\ref{fig3}). It is easy to see that only \(\overline{X}\), \(\{x, z\}\) and \(\{t, y\}\) are non singleton open balls in \((\overline{X}, \overline{d})\). Since each of these sets is a centered sphere, \((\overline{X}, \overline{d}) \in \mathbf{UGVL}\) by Theorem~\ref{t4.4}.
\end{example}

\begin{example}\label{ex5.4}
Let \((X, d) \) be infinite and let $d \colon X \times X \to \mathbb{R}^+$ be discrete. Then $(X,d)$ is an \(\mathbf{UGVL}\)-extension of itself, but there is no  minimal  $\mathbf{UGVL}$-extension of \((X, d)\).
\end{example}

To construct the minimal $\mathbf{UGVL}$-extensions of finite ultrametric space we will use the Gurvich---Vyalyi representing trees. Recall the procedure for constructing such trees.

With every finite nonempty ultrametric space \((X, d)\), we can associate a labeled rooted tree \( T(X,l) \) by the following rule (see \cite{PD2014JMS, DP2019PNUAA}). The root of \( T(X,l) \) is the set \(X\). If \(X\) is a one-point set, then \( T(X,l) \) is the rooted tree consisting of one node \(X\) with the label \(0\). Let \(|X| \geqslant 2\). According to Theorem~\ref{t3.3} the diametral graph $G_{X,d}$ is complete multipartite with the parts \( X_1, \ldots, X_k \), where all  \( X_1,\ldots ,X_k \) are open balls in $(X, d)$. In this case the root of the tree \( T(X,l) \) is labeled by \( \diam X > 0 \) and, moreover, \( T(X,l) \) has the nodes \(X_1\), \ldots, \(X_k\), \(k \geqslant 2\), of the first level with the labels
\begin{equation}\label{eq5.1}
l(X_i) = \diam X_i,
\end{equation}
\(i = 1, \ldots,k\). The nodes of the first level labeled by \(0\) are leaves, and those indicated by \(\diam X_i > 0\) are internal nodes of the tree \( T(X,l) \). If the first level has no internal nodes, then the tree \( T(X,l) \) is constructed. Otherwise, by repeating the above-described procedure with \(X_i\) corresponding to the internal nodes of the first level, we obtain the nodes of the second level, etc. Since \(X\) is a finite set, all vertices on some level will be leaves, and the construction of \( T(X,l) \) is completed.

We shall say that the labeled rooted tree \( T(X,l) \) is the \emph{representing tree} of \((X, d)\).

It can be shown that for any finite ultrametric space \( (X,d) \) the vertex set of the representing tree \( T = T(X, l) \) coincides with the set of all open balls of \( (X,d) \),
\begin{equation}\label{lbpm:e0}
V(T) = \mathbf{B}_X.
\end{equation}
(See, for example, Theorem~1.6 in \cite{DP2019PNUAA}).
Using Theorem~\ref{t1.4} and equality~\eqref{lbpm:e0} we see that defined above labeling $l \colon V(T) \to \mathbb{R}^+$ generate an ultrametric on the set $\mathbf{B}_X$. It is interesting to note that this ultrametric is the Hausdorff metric $d_H$ on $ \mathbf{B}_X$,
\begin{equation*}
d_H (B_1,B_2)= \max\{ \sup\limits_{x \in B_1} d(x,B_2), \sup\limits_{x \in B_2} d(x,B_1)\},
\end{equation*}
where
\begin{equation*}
d(x,B)= \inf\limits_{b \in B} d (x,b).
\end{equation*}

\begin{proposition}\label{p5.4}
    Let \((X, d)\) be a finite nonempty ultrametric space with the Gurvich--Vyalyi representing tree $ T(X,l)$. Then $(\mathbf{B}_X, d_H)$ is an $\mathbf{UGVL}$-ultrametric space generated by labeled tree  $T(l)=T(X,l) $.
\end{proposition}

For the proof see Theorem~2.5 in \cite{Dov2019}.

The next characterization of finite \textbf{UGVL}-spaces can be considered as a reformulation of Theorem~4.1 from \cite{Pet}.

\begin{theorem}\label{t5.5}
Let \((X, d)\) be a finite nonempty ultrametric space. Then \( (X,d) \in \mathbf{UGVL}\) if and only if every internal node of the representing tree  \(T(X,l) \) has at least one direct successor which is a leaf of \(T(X,l)\).
\end{theorem}

\begin{figure}[ht]
\begin{center}
\begin{tikzpicture}[scale=0.7]
\coordinate [label=left: $x$] (x) at (0,2);
\coordinate [label=right: $y$] (y) at (3,2);
\coordinate [label=right: $z$] (z) at (3,0);
\coordinate [label=left: $t$] (t) at (0,0);
\draw (x) -- node [above] {\(2\)} (y) -- node [right] {\(2\)} (z) -- node [below] {\(2\)} (t) -- node [left] {\(2\)} (x);

\draw (x) -- node [near start, below] {\(1\)} (z);
\draw (y) -- node [near start, below] {\(1\)} (t);

\node at (2, 3.5) {\((X, d)\)};

\draw [fill, black] (x) circle (1pt);
\draw [fill, black] (y) circle (1pt);
\draw [fill, black] (z) circle (1pt);
\draw [fill, black] (t) circle (1pt);
\end{tikzpicture}
\quad
\begin{tikzpicture}[
level 1/.style={level distance=\levdist,sibling distance=4cm},
level 2/.style={level distance=\levdist,sibling distance=2cm},
level 3/.style={level distance=\levdist,sibling distance=1cm},
hollow node/.style={circle,draw,inner sep=1.5}]
%\tikzstyle{every node}=[circle,draw]
    \node [hollow node, label={[label distance=1cm] left:\(T_X\)}] {2}
        child {
        node [hollow node] {1}
        child { node [circle,draw,inner sep=1,fill=black, label=below:{$\{x\}$} ] {} }
        child { node [circle,draw,inner sep=1,fill=black, label=below:{$\{z\}$} ] {} }
    }
    child {
        node [hollow node] {1}
        child { node [circle,draw,inner sep=1,fill=black, label=below:{$\{y\}$} ] {} }
        child { node [circle,draw,inner sep=1,fill=black, label=below:{$\{t\}$} ] {} }
 };
\end{tikzpicture}
\end{center}
\end{figure}

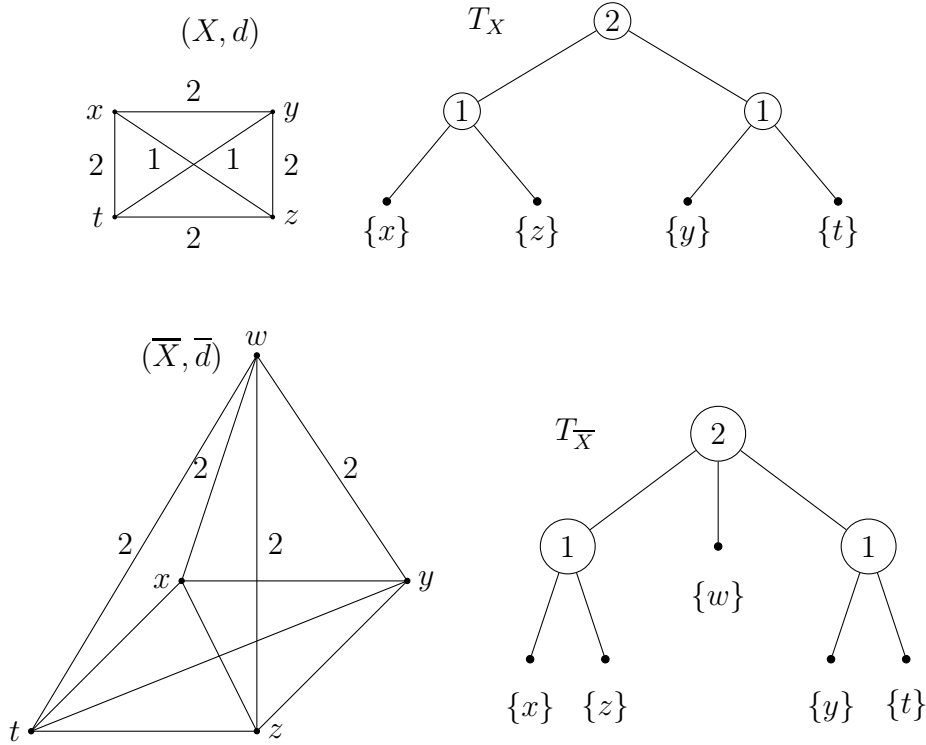
\begin{figure}[ht]
\begin{tikzpicture}[scale=1]
\coordinate [label=left: $x$] (x) at (0,2);
\coordinate [label=right: $y$] (y) at (3,2);
\coordinate [label=right: $z$] (z) at (1,0);
\coordinate [label=left: $t$] (t) at (-2,0);
\coordinate [label=above: $w$] (w) at (1,5);
\draw (x) -- (y) -- (z) -- (t) -- (x);

\draw (x) -- node [left] {\(2\)} (w) -- node [right] {\(2\)} (y);
\draw (t) --  node [left] {\(2\)}  (w) -- node [right] {\(2\)} (z);

\draw (x) -- (z);
\draw (y) -- (t);

\node at (0, 5) {\((\overline{X}, \overline{d})\)};

\draw [fill, black] (x) circle (1pt);
\draw [fill, black] (y) circle (1pt);
\draw [fill, black] (z) circle (1pt);
\draw [fill, black] (t) circle (1pt);
\draw [fill, black] (w) circle (1pt);
\end{tikzpicture}
\quad
\begin{tikzpicture}[
level distance=1.5cm,
level 1/.style={sibling distance=2cm},
level 2/.style={sibling distance=1cm} ]
\tikzstyle{every node}=[circle,draw]
    \node (Root) [label={[label distance=1cm] left:\(T_{\overline{X}}\)}] {2}
        child {
        node {1}
       child { node [circle,draw,inner sep=1,fill=black, label=below:{$\{x\}$} ] {} }
        child{ node [circle,draw,inner sep=1,fill=black, label=below:{$\{z\}$} ] {} }
    }
    child {
        node [inner sep=1, fill=black, label=below:{$\{w\}$} ] {}
    }
    child {
        node {1}
        child { node [circle,draw,inner sep=1,fill=black, label=below:{$\{y\}$} ] {} }
        child { node [circle,draw,inner sep=1,fill=black, label=below:{$\{t\}$} ] {} }
 };
\end{tikzpicture}
\begin{center}
\caption{The representing trees of the quadruple $(X,d)$ and the pyramid $(\overline{X}, \overline{d})$. The singleton $\{w\}$ is a leaf of the root of ${T_{\overline{X}}}$.}
\end{center}
\label{fig4}
\end{figure}

The next proposition and equality~\eqref{lbpm:e0} show that Theorems~\ref{t4.4} and \ref{t5.5} are really equivalent.

\begin{proposition}\label{p5.7}
  Let \((X, d)\) be a finite ultrametric space, $B \in \mathbf{B}_X$ be an internal node of the representing tree $T(X,l)$ and let $c$ be a point of $B$. Then the ball $B$ is a centered sphere with a center $c$ if and only if the singleton $\{c\}$ is a leaf of $T(X,l)$ and, simultaneously, a direct successor of the internal node $B$.
\end{proposition}

\begin{proof}
The validity of this assertion is checked using the above procedure for constructing $T(X,l)$ and Corollary~\ref{c3.3} with $Y=B$ and $\rho=d_{|B\times B}$.
\end{proof}

The proofs of the following two theorems can be found in \cite{DP2018pNUAA}.

\begin{theorem}\label{t1.10_from7}
Let \( (X, d) \) and \( (Y, \rho) \) be nonempty finite ultrametric spaces. Then the representing trees of these spaces are isomorphic as labeled rooted trees if and only if \( (X, d) \) and \( (Y, \rho) \) are isometric.
\end{theorem}

\begin{theorem}\label{lem1.11_from7}
Let \(T = T(r, l)\) be a finite labeled rooted tree with the root \(r\) and the labeling \(l \colon V (T) \to \mathbb{R}^{+} \). Then the following two conditions are equivalent.
\begin{enumerate}
\item For every \(u \in V (T)\) we have \( \delta^{+} (u) \neq 1\) and
$$
(\delta^{+}(u) = 0) \Leftrightarrow (l(u) = 0)
$$
and, in addition, the inequality
\begin{equation}
l(v) < l(u)
\end{equation}
holds whenever v is a direct successor of \(u\).
\item There is a finite ultrametric space \((X, d)\) such that the representing tree of \( (X,d)\) and \(T\) are isomorphic as labeled rooted trees.
\end{enumerate}
\end{theorem}

For every finite nonempty ultrametric space we denote by \( \Delta(X, d)\) the number of open balls which are not centered spheres in \((X, d)\),
\begin{equation}\label{ex5.3:e1}
\Delta(X, d) := \card \left(\mathbf{B}_{X, d} \setminus \mathbf{Cs}_{X, d}\right).
\end{equation}

The following theorem is the second main result of the paper.

\begin{theorem}\label{t5.4}
Let \((X, d)\) be a finite nonempty ultrametric space. Then the following statements hold:
\begin{enumerate}
\item Every \(\mathbf{UGVL}\)-extension \((Y, \rho)\) of \((X, d)\) satisfies the inequality
\begin{equation}\label{t5.4:e1}
\card Y \geqslant \Delta(X, d) + \card X.
\end{equation}
\item An \(\mathbf{UGVL}\)-extension \((Y, \rho)\) of \((X, d)\) is minimal if and only if the equality
\begin{equation}\label{t5.4:e2}
\card Y = \Delta(X, d) + \card X
\end{equation}
holds.
\item All minimal \(\mathbf{UGVL}\)-extensions of \((X, d)\) are isometric.
\end{enumerate}
\end{theorem}

\begin{proof}
Let $(Y, \rho)$ be an $\mathbf{UGVL}$-extension of $(X,d)$ and let $\Phi \colon X\to Y$ be an isometric embedding of $(X,d)$ in $(Y, \rho)$.
Write
$$
Y_1 \colon= \Phi(X) \ \textrm{ and } \ \rho_1 \colon = \rho_{|Y_1 \times Y_1}.
$$
Then $(Y, \rho)$ is an $\mathbf{UGVL}$-extension of $(Y_1, \rho_1)$, the equalities $\Delta(X,d)= \Delta(Y_1, \rho_1)$ and $\card X= \card Y_1$ hold, and, moreover, $(Y,\rho)$ is minimal for $(X,d)$ iff it is minimal for $(Y_1, \rho_1)$. Thus, without loss generality we may consider only these $\mathbf{UGVL}$-extensions of the space $(X,d)$ that are superspaces of this space.

$(i)$. Let us consider an arbitrary $\mathbf{UGVL}$-extension $(Y,\rho)$  of $(X,d)$. Since $X$ is finite, inequality \eqref{t5.4:e1} evidently holds if $Y$ is infinite.

Suppose that $(Y, \rho)$ is a finite ultrametric space, and define a mapping $F \colon \mathbf{B}_{X,d} \to \mathbf{B}_{Y, \rho}$ as
\begin{equation}\label{t5.7_p:e1}
F(B) \colon= \{ y \in Y \colon \rho (y,b) \leq \diam B \}
\end{equation}
where $b$ is an arbitrary given point of $B$.

We claim that $F$ is an injective mapping. Let us prove it. As in the proof of statement $(i)$ of Proposition~\ref{c3.11} we have the equality
\begin{equation}\label{t5.7_p:e2}
B = \{ x \in X \colon d (x,b) \leq \diam B \}
\end{equation}
for all $B \in \mathbf{B}_{X,d}$ and $b \in B$. Since $(X,d)$ is a subspace of $(Y, \rho)$, \eqref{t5.7_p:e1} and \eqref{t5.7_p:e2} imply
\begin{equation}\label{t5.7_p:e3}
B \subseteq F(B)
\end{equation}
for every $B \in \mathbf{B}_{X,d}$.

Let $B_1$ and $B_2$ be open balls in $(X,d)$ such that
\begin{equation}\label{t5.7_p:e4}
F(B_1) = F(B_2).
\end{equation}
We must show that
\begin{equation}\label{t5.7_p:e5}
B_1 = B_2.
\end{equation}
Inclusion \eqref{t5.7_p:e3} and equality \eqref{t5.7_p:e4} give us
\begin{equation}\label{t5.7_p:e6}
B_2 \subseteq F(B_1) \quad \textrm{and } \ B_1 \subseteq F(B_2).
\end{equation}
Let $b_i$ be an arbitrary point of $B_i$, $i=1,\, 2$. The equality $d=\rho_{|X\times X}$, \eqref{t5.7_p:e1} and \eqref{t5.7_p:e6} imply
$$
d(b_1, b_2) \leqslant \diam B_1
$$
and
$$
d(b_1, b_2) \leqslant \diam B_2.
$$
Hence the membership relations $b_1 \in B_2$ and $b_2 \in B_1$ are valid. Since $b_i$ is an arbitrary point of $B_i$, $i = 1, 2$, it implies
$$
B_1 \subseteq B_2 \quad \textrm{and} \quad B_2 \subseteq B_1.
$$
Equality \eqref{t5.7_p:e4} follows. Thus $F \colon \mathbf{B}_{X,d} \to \mathbf{B}_{Y,\rho}$ is injective.

Let us prove the validity of inequality \eqref{t5.4:e1}.

Since $(Y, \rho)$ belongs to the class $ \mathbf{UGVL}$, Theorem~\ref{t4.4} implies that a ball $F(B)$ is a centered sphere in $(Y, \rho)$ for every $B \in \mathbf{B}_{X,d}$. If $c$ is a center of the centered sphere $F(B) \in \mathbf{Cs}_{Y,\rho}$ and
\begin{equation}\label{t5.7_p:e7}
B \in \mathbf{B}_{X,d} \setminus \mathbf{Cs}_{X,d},
\end{equation}
then $c$ is not a point of the ball $B$,
\begin{equation}\label{t5.7_p:e8}
c \notin B.
\end{equation}
Indeed, since $(X,d)$ is a subspace $(Y, \rho)$, Definition~\ref{d1.2} and the membership $c \in B$ imply that $B$ is a centered sphere in $(X,d)$
\begin{equation*}
B \in \mathbf{Cs}_{X,d},
\end{equation*}
contrary to~\eqref{t5.7_p:e7}. Thus~\eqref{t5.7_p:e8} holds. By Proposition~\ref{p5.7}, the set $\{c\}$ is a leaf of $T_Y$ and, at the same time, a direct successor of the node $F(B)$ of the representing tree $T_Y$. Since different nodes of $T_Y$ do not have common direst successors, inequa\-li\-ty\eqref{t5.4:e1} follows from the injectivity of mapping $F \colon \mathbf{B}_{X,d} \to \mathbf{B}_{Y, \rho}$.

$(ii)$. Let $(Y, \rho)$ be an $\mathbf{UGVL}$-extension of $(X,d)$ and let equa\-li\-ty~\eqref{t5.4:e2} hold. Then $Y$ is finite set. Since every proper subset $Y_0$ of the finite set $Y$ satisfies the inequality
$$
\card Y_0 < \card Y,
$$
equality\eqref{t5.4:e2} implies
$$
\card Y_0 < \Delta (X,d) + \card X.
$$
Consequently $(Y_0, \rho_{|Y_0 \times Y_0})$ is not an $\mathbf{UGVL}$-extension of $(X,d)$ by statement $(i)$. Thus if equality~\eqref{t5.4:e2} holds, then $(Y, \rho)$ is  a minimal $\mathbf{UGVL}$-extension of $(X,d)$.

Let $(Y, \rho)$ be a minimal $\mathbf{UGVL}$-extension of $(X,d)$. We will prove that $(Y, \rho)$ satisfies equality~\eqref{t5.4:e2}.

First of all prove that any minimal $\mathbf{UGVL}$-extension $W(\lambda)$ of a given finite nonempty ultrametric space $S(\delta)$ is also a finite ultrametric space. By Proposition~\ref{p4.1} it suffices to consider the case
\begin{equation}\label{t5.7_p:e9}
\card S \geqslant 4.
\end{equation}
Let $T=T(l)$ be a labeled tree generating the space $(W, \lambda)$. As was noted in the first part of the proof we may also suppose that $W \supseteq S$. Hence $S$ is a subset of the vertex set $V(T)$. The union of all paths connecting in $T$ different points of $S$ is a finite subtree $T^S$ of $T$. It should be noted that this union is nonempty due to~\eqref{t5.7_p:e9}. Let us define a labeling $l^S \colon V (T^S) \to \mathbb{R}^+$ as $l^S = l_{|V(T^S)}$. Then the ultrametric space $(V(T^S), \delta^S)$, with
$$
\delta^S= \lambda_{|V(T^S) \times V(T^S)},
$$
is a superspace for $(S,\delta)$, and a finite subspace of $(W,\lambda)$, and, moreover, $(V(T^S), \delta^S)$ is generated by labeled tree $T^S(l^S)$. Consequently $(V(T^S), \delta^S) \in \mathbf{UGVL}$ holds, that implies
$$
(V(T^S), \delta^S) = (W, \lambda)
$$
due to the minimality of $(W, \lambda)$.

Thus $(Y,\rho)$ is a finite ultrametric space.

Let $\mathbf{B}^I_{X,d}$ and $\mathbf{B}^I_{Y,\rho}$ be defined as
\begin{equation}\label{t5.7_p:e10}
\mathbf{B}^I_{X,d} \colon = \{ B \in \mathbf{B}_{X,d} \colon \diam B>0 \},
\end{equation}
and
\begin{equation}\label{t5.7_p:e11}
\mathbf{B}^I_{Y,\rho} \colon = \{ B \in \mathbf{B}_{Y,\rho} \colon \diam B>0 \}.
\end{equation}
We claim that the equality
\begin{equation}\label{t5.7_p:e12}
\diam (B \cap X) = \diam B
\end{equation}
holds for every $B \in \mathbf{B}^I_{Y, \rho}$. For the case  $B=Y$, equality~\eqref{t5.7_p:e12} can be written as
\begin{equation}\label{t5.7_p:e13}
\diam (Y \cap X) = \diam Y.
\end{equation}
Suppose, on the contrary, that
$$
\diam (Y \cap X) <  \diam Y.
$$
Then, by Theorem~\ref{t3.3}, there is an open ball $B_r(y_0) \in \mathbf{B}_{Y, \rho}$ with $r=\diam Y$ such that
\begin{equation}\label{t5.7_p:e15}
X \subseteq B_r(y_0).
\end{equation}
The diametrical graph $G_{Y, \rho}$ is complete multipartite, hence the set $Y \setminus B_r(y_0)$ is nonvoid,
\begin{equation}\label{t5.7_p:e16}
\card (Y \setminus B_r(y_0)) >0.
\end{equation}
The Gurvich--Vyalyi representing tree $T_{B_r(y_0)}$ is a subtree of the representing tree $T_Y$ with the vertex set $V(T_{B_r(y_0)})$  consisting of all successor of the vertex $B_r(y_0)$ in $T_Y$. By Theorem~\ref{t5.5} the ultrametric space $\left( B_{r(y_0)}, \rho_{|B_r(y_0) \times B_r(y_0)}  \right)$ belongs to the class $\mathbf{UGVL}$. Hence it is an $\mathbf{UGVL}$-extension of $(X,d)$ by \eqref{t5.7_p:e15}. Now inequality \eqref{t5.7_p:e16} shows that $(Y, \rho)$ is not a minimal $\mathbf{UGVL}$-extension of $(X,d)$ contrary to the supposition. Thus \eqref{t5.7_p:e13} holds.

Let $B_1^0, B_2^0, \ldots, B_n^0$ be the paths of diametrical graph $G_{Y, \rho}$. Suppose that $B_1^0 \in \mathbf{B}_{Y, \rho}^I$ and $\diam (B^0_1 \cap X) < \diam B_2^0$. Then using Theorem~\ref{t3.3} and Corollary~\ref{c2.4} we can find $B_1^1 \in \mathbf{B}^I_{Y, \rho}$ such that $B_1^1$ is a path of the complete multipartite graph $G_{B_1^0,  \rho_{| B_1^0 \times B_1^0}}$ satisfying
\begin{equation}\label{t5.7_p:e17}
 B_1^1 \supseteq B_1^0 \cap X.
 \end{equation}
 Write
 \begin{equation}\label{t5.7_p:e18}
Y_1 \colon = (Y \setminus B_1^0) \cup B_1^1.
\end{equation}
Then $B_1^0, B_2^0, \ldots, B_n^0$ be the paths of the diametrical graph $G_{Y_1, \rho_{|Y_1 \times Y_1}}$ and, consequently,
$$\mathbf{B}_{Y_1, \rho_{|Y_1 \times Y_1}} \subseteq \mathbf{B}_{Y, \rho}$$
holds by Corollary~\ref{c2.4}. Now using Theorem~\ref{t5.5} we see that $(Y_1, \rho_{|Y_1 \times Y_1}) \in \mathbf{UGVL}$ and $X \subseteq Y_1$ by \eqref{t5.7_p:e17}--\eqref{t5.7_p:e18}. Consequently $(Y_1, \rho_{|Y_1 \times Y_1})$ is an $\mathbf{UGVL}$-extension of $(X,d)$. Since $B_1^1$ is a proper subset of $B_2^0$, equa\-li\-ty~\eqref{t5.7_p:e18}  implies
$$
\card Y_1 < \card Y.
$$
Hence $(Y, \rho)$ is not a minimal $\mathbf{UGVL}$-extension of $(X,d)$, contrary to the supposition. Thus the equality
$$
\diam (B_1^0 \cap X) = \diam B_1^0
$$
holds.

Similarly we obtain
\begin{equation}\label{t5.7_p:e19}
\diam (B_i^0 \cap X) = \diam B_i^0
\end{equation}
whenever $i=2, \ldots, n$ and $B_i^0 \in \mathbf{B}_{Y, \rho}^I$. Thus \eqref{t5.7_p:e12} holds for all internal nodes of $T_X$ having the first level. Now considering the diametrical graphs $G_{B_1^0, \rho_{|B_1^0 \times B_1^0}}, \ldots, G_{B_n^0, \rho_{|B_n^0 \times B_n^0}}$ instead of the graph $G_{Y, \rho}$ we obtain~\eqref{t5.7_p:e12} for the nodes of the second level and so on. Consequently \eqref{t5.7_p:e12} holds for every $B \in \mathbf{B}_{Y, \rho}^I$, due to the finiteness of $(Y, \rho)$ and the equality
$$
V(T_Y) = \mathbf{B}_{Y,\rho}.
$$

Let $T_X$ and $T_Y$ be the representing trees of $(X,d)$ and, respectively, $(Y, \rho)$. Using Proposition~\ref{p5.7} we can find a subset $Y_0$ of $Y$ such that for every $y_0 \in Y_0$, the singleton $\{y_0\}$ is a leaf of an internal node $F(B)$ with
$$
B \in \mathbf{B}_{X,d} \setminus \mathbf{Cs}_{X,d}
$$
and, in addition, for different points $y_1, y_2 \in Y_0$, the singletons $\{y_1\}$ and $\{y_2\}$  are leaves of different internal nodes of $T_Y$. Removing from $T_Y$ all leaves of type $\{y\}$  for $y \in Y \setminus (X \cup Y_0)$ we obtain a labeled rooted subtree $T^1$ of $T_Y$. Let us define a subset $Y_1$ of $Y$ as
$$
Y_1 \colon = X \cup Y_0.
$$
Using Theorems~\ref{t1.10_from7} and \ref{lem1.11_from7} we can prove that $T^1$ and the representing tree $T_{Y_1}$ of the ultrametric space $(Y_1, \rho_{|Y_1 \times Y_1})$ are isomorphic as labeled rooted trees.
We only note that
$$
\diam B = \diam (B \cap X)
$$
holds for every $B \in \mathbf{B}^I_{Y_1, \rho_{|Y_1 \times Y_1}}$ because \eqref{t5.7_p:e12} and
$$
B \cap X \subseteq B \cap Y_1 \subseteq B
$$
are valid for every $B \in \mathbf{B}^I_{Y, \rho}$.
If follows directly from the definition of $T^1$ that every internal node of $T^{1}$ has a leaf. Since $T^1$ and $T_{Y_1}$ are isomorphic as labeled rooted trees, every internal node of $T_{Y_1}$ also has a leaf.
Consequently $(Y_1, \rho_{|Y_1 \times Y_1})$ is an $\mathbf{UGVL}$-extension of $(X,d)$. The sets $X$  and $Y_0$ are disjoint and the equality
$$
\card Y_0 = \card (\mathbf{B}_{X,d} \setminus \mathbf{Cs}_{X,d})
$$
holds. Consequently we have
$$
\card Y_1 = \card X + \card Y_0 = \card X + \Delta(X,d).
$$
To complete the proof of equality~\eqref{t5.4:e2} is suffices to note that
$$
(Y, \rho)= (Y_1, \rho_{|Y_1 \times Y_1})
$$
because $(Y, \rho)$ is a minimal $\mathbf{UGVL}$-extension of $(X,d)$, and $(Y_1, \rho_{|Y_1 \times Y_1})$ is an $\mathbf{UGVL}$-extension of $(X,d)$ and $Y_1 \subseteq Y$.

$(iii)$. Let $(Y_1, \rho_1)$ and $(Y_2, \rho_2)$ be minimal $(\mathbf{UGVL})$-extensions of $(X,d)$. We must prove that $(Y_1, \rho_1)$ and $(Y_2, \rho_2)$ are isometric as metric spaces. It was noted above that we can consider only the case when
$$
Y_1 \supseteq X \quad \textrm{ and } \quad Y_2 \supseteq X.
$$
It was noted at the final part of the proof of statement $(ii)$ that, for every minimal $\mathbf{UGVL}$-extension $(Y, \rho)$ of $(X,d)$, the representing tree $T_Y$ is isomorphic to the labeled rooted tree which can be obtained from $T_X$ by gluing a leaf to each vertex $B \in \mathbf{B}_{X,d} \setminus \mathbf{Cs}_{X,d}$.

Hence the representing trees $T_{Y_1}$ and $T_{Y_2}$ are isomorphic. Consequently $(Y_1, \rho_1)$ and $(Y_2, \rho_2)$ are isometric by Theorem~\ref{t1.10_from7}.
\end{proof}

Let $\mathfrak{M}$ be a set of finite nonempty ultrametric spaces. An $\mathbf{UGVL}$-space $(Y, \rho)$ is said to be $\mathfrak{M}$-\emph{universal} if $(Y, \rho)$ is an $\mathbf{UGVL}$-extension for every $(X,d) \in \mathfrak{M}$.

We say that $(Y, \rho)$ is \emph{minimal $\mathfrak{M}$-universal} if for every proper $\mathbf{UGVL}$-subspace $(Y_1, \rho_{|Y_1 \times Y_1})$ of $(Y, \rho)$ there is $(X,d) \in \mathfrak{M}$ such that  $(Y_1, \rho_{|Y_1 \times Y_1})$ is non an $\mathbf{UGVL}$-extension of $(X,d)$.

\begin{problem}
    Find condition under which $\mathfrak{M}$ admits a minimal universal $\mathbf{UGVL}$-extension.
\end{problem}

Some results connected to minimal universal metric spaces can be found \cite{BDKP}.

\section{Acknowledgments}

Oleksiy Dovgoshey was supported by the Finnish Society of Sciences and Letters (Project
``Intrinsic Metrics of Domains in Geometric Function Theory and Graphs'').

\providecommand{\bysame}{\leavevmode\hbox to3em{\hrulefill}\thinspace}
\providecommand{\MR}{\relax\ifhmode\unskip\space\fi MR }
\providecommand{\MRhref}[2]{%
  \href{http://www.ams.org/mathscinet-getitem?mr=#1}{#2}
}
\providecommand{\href}[2]{#2}


\begin{thebibliography}{10}

\bibitem{BDK2022TaAoG}
V.~Bilet, O.~Dovgoshey, and Y.~Kononov, \emph{Ultrametrics and complete
  multipartite graphs}, Theory and Applications of Graphs \textbf{9} (2022),
  no.~1, Article 8.

\bibitem{BDKP}
V.~Bilet, O.~Dovgoshey, M.~K\"{u}\c{c}\"{u}kaslan, and E.~Petrov, \emph{Minimal
  universal metric spaces}, Ann. Acad. Sci. Fenn. Math. \textbf{42} (2017),
  1019--1064.

\bibitem{Diestel2017}
R.~Diestel, \emph{{Graph Theory}}, fifth ed., Graduate Texts in Mathematics,
  vol. 173, Springer, Berlin, 2017.

\bibitem{DDP2011pNUAA}
D.~Dordovskyi, O.~Dovgoshey, and E.~Petrov, \emph{Diameter and diametrical
  pairs of points in ultrametric spaces}, p-adic Numbers Ultrametr. Anal. Appl.
  \textbf{3} (2011), no.~4, 253--262.

\bibitem{Dov2019}
O.~Dovgoshey, \emph{Finite ultrametric balls}, p-adic Numbers Ultrametr. Anal.
  Appl. \textbf{11} (2019), no.~3, 177--191.

\bibitem{Dov2020TaAoG}
\bysame, \emph{Isomorphism of trees and isometry of ultrametric spaces}, Theory
  and Applications of Graphs \textbf{7} (2020), no.~2, Article 3.

\bibitem{DK2022AC}
O.~Dovgoshey and M.~K\"{u}\c{c}\"{u}kaslan, \emph{Labeled trees generating
  complete, compact, and discrete ultrametric spaces}, Ann. Comb. \textbf{26}
  (2022), 613--642.

\bibitem{DP2018pNUAA}
O.~Dovgoshey and E.~Petrov, \emph{{From isomorphic rooted trees to isometric
  ultrametric spaces}}, p-adic Numbers Ultrametr. Anal. Appl. \textbf{10}
  (2018), no.~4, 287--298.

\bibitem{DP2019PNUAA}
\bysame, \emph{{Properties and morphisms of finite ultrametric spaces and their
  representing trees}}, p-adic Numbers Ultrametr. Anal. Appl. \textbf{11}
  (2019), no.~1, 1--20.

\bibitem{Pet}
E.~Petrov, \emph{Hereditary properties of finite ultrametric spaces}, J. Math.
  Sci. \textbf{264} (2022), no.~4, 423--440, Translation from Ukr. Mat. Visn.
  19(2):213--236, 2022.

\bibitem{PD2014JMS}
E.~Petrov and A.~Dovgoshey, \emph{{On the {G}omory-{H}u inequality}}, J. Math.
  Sci. \textbf{198} (2014), no.~4, 392--411, Translation from Ukr. Mat. Visn.
  10(4):469--496, 2013.

\bibitem{Sch1985}
W.~H. Schikhof, \emph{{Ultrametric Calculus. An Introduction to p-Adic
  Analysis}}, Cambridge University Press, 1985.

\bibitem{Sea2007}
M.~\'{O}. Searc\'{o}id, \emph{{Metric Spaces}}, Springer---Verlag, London,
  2007.

\end{thebibliography}
\end{document}